\documentclass{article}
\usepackage[utf8]{inputenc}
\usepackage{amsmath,amsthm,amssymb}
\usepackage{amsopn}
\usepackage{enumitem}
\usepackage{ytableau}
\usepackage{fancyhdr}
\fancypagestyle{footer}{\fancyhf{}\fancyfoot[L]{
L. J. Jolliffe, Department of Pure Mathematics and Mathematical Statistics, Centre for Mathematical Sciences, University of Cambridge, Wilberforce Road, Cambridge, CB3 0WB, United Kingdom\\
\emph{E-mail address:}  ljj33@cam.ac.uk\\
S. Martin, Department of Pure Mathematics and Mathematical Statistics, Centre for Mathematical Sciences, University of Cambridge, Wilberforce Road, Cambridge, CB3 0WB, United Kingdom}}
\theoremstyle{plain}
\usepackage{titling}
\usepackage{tikz}
\usepackage{cleveref}
\newtheorem{theorem}{Theorem}
\newtheorem{conj}[theorem]{Conjecture}

\newtheorem{prop}[theorem]{Proposition}
\newtheorem{lemma}[theorem]{Lemma}
\newtheorem{cor}[theorem]{Corollary}
\newtheorem*{remark}{Remark}
\newtheorem*{example}{Example}
\newtheorem*{lemma*}{Lemma}
\newtheorem*{claim}{Claim}
\newtheorem{definition}[theorem]{Definition}
\crefalias{prop}{proposition}
\crefalias{cor}{corollary}
\begin{document}
\title{On the Schaper Numbers of Partitions}
\author{Liam Jolliffe \and Stuart Martin}
\date{\today}
\maketitle
\begin{abstract}
One of the most useful tools for calculating the decomposition numbers of the symmetric group is Schaper's sum formula \cite{Schaper}. The utility of this formula for a given Specht module can be improved by knowing the Schaper Number of the corresponding partition. Fayers \cite{Fayers} gives a characterisation of those partitions whose Schaper number is at least two. In this paper we shall demonstrate how this knowledge can be used to calculate some decomposition numbers before extending this result with the hope of allowing more decomposition numbers to be calculated in the future. For $p=2$ we shall give a complete characterisation of partitions whose Schaper number is at least three, and those whose Schaper number at least four. We also present a list of necessary conditions for a partition to have Schaper number at least three for odd primes and a conjecture on the sufficiency of these conditions.
\end{abstract}
\newpage
\section{Introduction}

Our notation is taken from James' book \cite{James}, to which we refer the reader for more information on the representatiton theory of the symmetric group, and from Fayers \cite{Fayers}. The theory of Scaper layers was introduced in Schaper's thesis \cite{Schaper}, but is best found in English in \cite{Donkin}, \cite{qSchaper} and \cite{Kunzer_Nebe}.  Recall, given $\lambda\vdash n$, we define a $\lambda$-\emph{tableau} to be a bijection between $[\lambda]$, the \emph{Young diagram} of shape $\lambda$, and $[n]$. Two tableaux $s$ and $t$ are \emph{row equivalent}, $s\sim_{row}t$, if the entries in each row of $s$ are the same as the entries in the corresponding rows of $t$. \emph{Column equivalence}, $s\sim_{col}t$, is definied similarly and the column stabiliser of $t$ is defined to be the set $C(t)=\lbrace \sigma \in \mathcal{S}_n \mid\sigma t \sim_{col} t \rbrace$. A $\lambda$-\emph{tabloid} is a row equivalence class of $\lambda$-tableaux, and will be denoted by writing the tableau in braces, $\lbrace s \rbrace$, or by drawing the Young diagram without vertical lines separating boxes. For a ring $R$, the $R$ span of all $\lambda$ tabloids is the \emph{Young module}, $M_R^\lambda$. We have an inner product on this space by linearly extending the following: 
$$\langle \lbrace s \rbrace , \lbrace t \rbrace \rangle = 
\begin{cases}
1 & \text{if } s\sim_{row}t\\
0 & \text{otherwise}
\end{cases}.$$ 
We define the \emph{column symmetriser} of $t$ to be the element of the group algebra, $R\mathcal{S}_n$, given by 
$$\kappa_t =\sum_{\sigma\in C(t)}(-1)^\sigma \sigma,$$ and define the \emph{polytabloid} $e_t=\kappa_t \lbrace t\rbrace$. The \emph{Specht module} $S_R^\lambda\subseteq M_R^\lambda$ is the $R$ span of polytabloids; In fact $S_R^\lambda=\langle e_t \mid t\in \text{std}(\lambda)\rangle_R$, where $\text{std}(\lambda)$ is the set of all \emph{standard} $\lambda$-tableaux, that is tableaux whose entries are increasing across rows and down columns. The Specht modules are a complete set of non-isomorphic irreducible $\mathbb{C}\mathcal{S}_n$ modules, but over fields of positive characteristic they are not necessarily irreducible. In this case the irreducible modules are the modules $D^\lambda=\frac{S^\lambda}{S^\lambda\cap (S^\lambda)^\perp}$ where $\lambda$ is $p$-regular, that is does not have $p$ consecutive rows of the same length (a $p$-\emph{singularity}), and orthogonality is with respect to the inner product. An important problem in the representation theory of the symmetric group is calculation of the composition multiplicity, $[S^\lambda : D^\mu]$, of the irreducible module $D^\mu$ in the Specht module $S^\lambda$.  

For a given Specht module $S^\lambda_\mathbb{Z}$ and prime $p$ we define the submodule $$S^\lambda_i=\{x\in S^\lambda_\mathbb{Z} : p^i \mid \langle x,y \rangle \forall y \in S^\lambda_\mathbb{Z}  \}$$
and denote by $\Bar{S^\lambda_i}$ its reduction mod $p$ to obtain the \emph{Schaper filtration}: 
$$S^\lambda_{\mathbb{F}_p}=\Bar{S^\lambda_0}\geq \Bar{S^\lambda_1}\geq \Bar{S^\lambda_2}\geq \cdots.$$
All composition factors of $S^\lambda_{\mathbb{F}_p}$ must all appear in the quotients of this filtration and hence studying this filtration would reveal the decomposition numbers for the symmetric groups. Unfortunately, the layers of this filtration are not known in general, but despite this we are able to use combinatorial tools to calculate an upper bound for the decomposition number $[S^\lambda : D^\mu]$. 

Let $\lambda\vdash n$ and define $H(\lambda)$ be the set of triples $(g,h,\nu)$ where $\nu\unlhd\lambda$ is a partition of $n$ and $g$ and $h$ are hooks of $Y(\lambda)$ and $Y(\nu)$ respectively, such that removing the corresponding rim-hooks leaves the same partition $Y(\lambda\backslash g)=Y(\nu \backslash h)$.  

\begin{theorem}[Schaper's Sum Formula]\label{Schaper's sum formula}

Let $\mu$ be a $p$-regular partition not equal to $\lambda$, then $$\sum_{i\ge 1}[\Bar{S^\lambda_{i}} : D^\mu]=\sum_{(g,h,\nu)\in H(\lambda)} \alpha_{(g,h,\nu)} [S^\nu : D^\mu],$$ where the coefficients are given by $\alpha_{(g,h,\nu)}=\nu_p(\mid g \mid)(-1)^{l(g)+l(h)+1}$.
\end{theorem}
The top factor of this filtration $\Bar{S^\lambda_0} / \Bar{S^\lambda_1}$ is $D^\lambda$ if the partition $\lambda$ is $p$-regular, otherwise is zero and hence the sum formula gives an upper bound on $[S^\lambda : D^\mu]$ for $\mu\lhd \lambda$, as any composition factor isomorphic to $ D^\mu$ must appear in a quotient further along the filtration.  
\begin{definition}
The $i$th \emph{Schaper layer} of the Specht module $S^\lambda$ is $$L_i=\Bar{S^\lambda_i} / \Bar{S^\lambda_{i+1}}.$$ We say that the $k$th Schaper layer of $S^\lambda$ is the \emph{top layer} if $k$ is the least such integer such that $L_k\ne 0$. The integer $k$ will be denoted by $\nu_p(\lambda)$ and shall be called the ($p$) \emph{Schaper number} of $\lambda$. 
\end{definition}
An irreducible module appearing in the $i$th layer is counted by the formula $i$ times as it is a composition factor of $\Bar{S^\lambda_j}$ for all $j\leq i$, and hence knowing which layer is the top layer allows us to improve the upper bound for the decomposition numbers obtained from Schaper's sum formula.

\begin{cor}\label{Sum formula cor}
$$[S^\lambda:D^\mu]\le \frac{\sum_{i\ge 1}[\Bar{S^\lambda_{i}} : D^\mu]}{\nu_p(\lambda)}.$$
\end{cor}

Fayers \cite{Fayers} showed that Schaper numbers of partitions are superadditive in the following sense: 
\begin{prop}\label{(lambda*mu)}
Let $\lambda = (\lambda_1,\dots,\lambda_r) \vdash n$ and $\mu=(\mu_1,\dots,\mu_s) \vdash m$. Then $\nu_p(\lambda\star\mu)\ge \nu_p(\lambda)+\nu_p(\mu)$, where $\lambda\star\mu$ is the partition of $n+m$ which has rows of lengths $\lambda_1,\dots,\lambda_r,\mu_1,\dots,\mu_{s-1}$ and $\mu_s$.
\end{prop}
This result is reminiscent of Donkin's generalisation \cite{Donkin} of the principle of row removal \cite{Jamesrowremoval} and is useful in determining lower bounds on the Schaper number of a partition.

\section{The Schaper Number of $\lambda$}

In this section we shall turn to characterising partitions $\lambda$ with a certain Schaper number. We use a number of results and techniques due to Fayers \cite{Fayers}, which are stated here.
A corollary of the following theorem of James \cite{James} tells us that $\nu_p(\lambda)\ge 1$ if and only if $\lambda$ is $p$-singular:
\begin{theorem}\label{James}
Suppose $\lambda$ has $z_j$ parts equal to $j$ for each $j$. Then 
$$\nu_p(\prod^{\infty}_{1}z_j!)\le \nu_p(\lambda) \le \nu_p(\prod^{\infty}_{1}(z_j!)^j).$$
\end{theorem}

We shall use the graph-theoretic approach introduced by Fayers \cite{Fayers}. 
Recall if $s$ and $t$ are row equivalent $\lambda$-tableaux we define the graph $G=G(s,t)$
as follows: the vertex set of $G$ is $ \lbrace s_1,s_2,\dots,s_{\lambda'_1}, t_1,t_2,\dots,t_{\lambda'_1} \rbrace $ and the edge set is $\lbrace e_1, \dots, e_n \rbrace$ and the edge $e_l$ goes from $s_i$ to $t_j$ if $l$ appears in column $i$ of $s$ and column $j$ of $t$. We consider colourings of $G(s,t)$ with colours $c_1,\dots,c_{\lambda'_1}$ and we call such a colouring admissible if for each $l$ there is precisely one edge of colour $c_l$ incident on each of the vertices $s_1,\dots,s_{\lambda'_l},t_1,\dots,t_{\lambda'_l}$. The set of all admissible colourings of $G$ will be denoted $A(G)$. Observe there is a bijection between the admissible colourings of $G$ and pairs $(u,v)$ of $\lambda$-tableaux with $s\sim_{\text{col}}u\sim_{\text{row}}\sim_{\text{col}}t$. This correspondence is given by colouring the edge $e_i$ with colour $i$ if it appears in row $i$ of $u$, or equivalently row $i$ of $v$. Observe each admissible colouring induces a permutation of $\lbrace 1,2,\dots,\lambda'_l\rbrace$ for each $l$ by sending $i$ to $j$ if there is an edge from $s_i$ to $t_j$ of colour $l$. If $u$ and $v$ are the corresponding tableaux then this permutation which takes the $l$th row of $u$ to th $l$th row of $v$. Define the product of all of the signatures of these permutations for all $l$ to be the signature of the colouring, $(-1)^C$, and observe that as $(-1)^C=(-1)^{\pi_{uv}}=(-1)^{\pi_{st}}(-1)^{\pi_{us}}(-1)^{\pi_{tv}}$ we get the following result \cite{Fayers}:
\begin{prop}
$$\sum_{C\in A(G)}(-1)^C=(-1)^{\pi_{st}} \langle e_s , e_t\rangle .$$
\end{prop}
Fayers uses this approach to prove a result reminiscent of principle of column removal \cite{Jamesrowremoval}:
\begin{prop}\label{columnremoval}
Let $\hat{\lambda}$ be the partition whose Young diagram is obtained by removing the first column of the Young diagram for $\lambda$. Then $\nu_p(\lambda) \ge \nu_p(\hat{\lambda})$. 
\end{prop} 
An important consequence of the proof of \Cref{columnremoval} is the following: 
\begin{prop}\label{p edges}
Let $s$ and $t$ be $\lambda$-tableau. If there are $m$ edges from $s_1$ to $t_1$ in $G(s,t)$ then $\langle e_s , e_t\rangle$ is divisible by $m!p^{\nu_p(\hat{\lambda})}$, where $\hat{\lambda}$ is the partition whose Young diagram is obtained by removing the first column of the Young diagram for $\lambda$.
\end{prop}
 
This graph theoretic approach allows Fayers to go further than James, and characterise all of Specht modules whose Schaper number at least two:
\begin{theorem}\label{Fayers}
Let $\lambda \vdash n$. Then $\nu_p(\lambda)\ge 2$ if and only if one of the following hold:
\begin{enumerate}
\item $\lambda$ is doubly $p$-singular; that is there exists $i,j$ with $i \ge j+p$ and $\lambda_i=\lambda_{i+p-1}$ and $\lambda_j=\lambda_{j+p-1}$.
\item There exist $i$ such that $\lambda_i\le\lambda_{i+2p-2}+1$ and $\lambda_{i+p-1}\ge 2$.
\end{enumerate}
\end{theorem}
This result, together with \Cref{(lambda*mu)} immediately gives the corollary below. The reader can see the obvious extension of this and should now be able to construct partitions with arbitrarily large Schaper numbers.
\begin{cor}\label{firstconditions}
Let $\lambda \vdash n$. Then $\nu_p(\lambda)\ge 3$ if one of the following hold:
\begin{enumerate}
\item $\lambda$ is triply $p$-singular.
\item There exist $i,j$ with $\{i,\dots, i+2p-2\}\cap\{j,\dots,j+p-1\}=\emptyset$ such that $\lambda_i\le\lambda_{i+2p-2}+1$ and $\lambda_{i+p-1}\ge 2$ and $\lambda_j=\lambda_{j+p-1}$.
\end{enumerate}
\end{cor} 

Before continuing we shall give an example of how decomposition numbers can be calculated using \Cref{Fayers}:
\begin{example}
Let $p=2$ and consider the block of $\mathcal{S}_{{13}}$ containing all Specht modules $S^\lambda$ where $\lambda$ has 2-core $(2,1)$.
Assume that the decomposition numbers are known for $\mathcal{S}_n$ where $n<13$. Using column elimination \cite{Donkin} and by observing the linear relations between the ordinary characters of $\mathcal{S}_{13}$ on $2$-regular classes we can compute the first part of the first column of the decomposition matrix below:

$$\begin{array}{cc}
 & (12,1)\\
(12,1) & 1\\
(10,3) & 0\\
(10,1^3) & 1\\
(8,5) & 1\\
(8,3,2) & x\\
(8,3,1^2) & x+1\\
(8,2^2,1) & x
\end{array}$$
Schaper's sum formula, \Cref{Schaper's sum formula}, tells us $$[S^{(8,3,2)}:D^{(12,1)}]\le [S^{(12,1)}:D^{(12,1)}]+[S^{(8,5)}:D^{(12,1)}]$$
and
\begin{align*}
[S^{(8,2,2,1)}:D^{(12,1)}]\le &-2[S^{(12,1)}:D^{(12,1)}]+[S^{(10,1^3)}:D^{(12,1)}]\\&-2[S^{(8,5)}:D^{(12,1)}]+2[S^{(8,3,2)}:D^{(12,1)}]\\&+[S^{(8,3,1,1)}:D^{(12,1)}].
\end{align*}
That is $x\le 2$ and $x\le3x-2$. \Cref{Fayers} allows us to improve the second inequality, as we know that the Schaper number of $(8,2,2,1)$ is at least two. Thus, using \Cref{Sum formula cor}, the second inequality becomes $x\le\frac{3x-2}{2}$ and we conclude that $[S^{(8,2,2,1)}:D^{(12,1)}]=2$.
\end{example}
Of course this decomposition number can be calculated using other techniques, but this calculation demonstrates how a better understanding of Schaper numbers may lead to new decomposition numbers for the symmetric group. The following lemma gives us another way of constructing partitions of large Schaper number. 
\begin{lemma}\label{abc}
Let $\lambda=((x+1)^a,x^b,(x-1)^c)\vdash n$. Then $$\nu_p(\lambda)\ge \nu_p((x^{a+b+c}))-\nu_p({{a+b+c}\choose {a,b,c}})-\nu_p(c!).$$
Here $\nu_p$ is used to denote both the usual $p$-adic valuation (of an integer) and the Schaper number (of a partition).
\end{lemma}
We shall prove this by induction on $a$, where the base case $a=0$ follows from the proof of \cite[3.12]{Fayers} :

\begin{lemma}[Base case]\label{base case}
Let $\lambda=(x^a,(x-1)^b)\vdash n$. Then $$\nu_p(\lambda)\ge \nu_p((x^{a+b}))-\nu_p({{a+b}\choose {a}}).$$
\end{lemma}
\begin{proof}
We shall prove this by induction on $a$, with the base case $a=0$ being trivial. Suppose $a>0$ and let $s$ and $t$ be row equivalent $\lambda=(x^a,(x-1)^b)$-tableau. Construct the graph $G=G(s,t)$ and suppose there is an edge $e$ from $s_x$ to $t_x$ in $G$. Permuting the colours $c_1,\dots,c_a$ gives rise to a faithful and signature preserving action of $\mathcal{S}_a$ on the admissible colourings $A(G)$ of $G$. Summing the signatures of all the admissible colourings of $G$ in which the edge $e$ has colour $a$ gives $\frac{\langle e_s, e_t \rangle}{a}$. Deleting the edge $e$ results in a graph $G(s',t')$ for some $(x^{a-1},(x-1^{b+1}))$-tableau $s'$ and $t'$. There is an obvious correspondence between admissible colourings of $G(s',t')$ and those admissible colourings of $G$ where $e$ has the colour $c_a$. Observe that, as this correspondence preserves the signature, the sum of all admissible colourings of $G$ in which the edge $e$ has colour $a$ is $\langle e_{s'}, e_{t'} \rangle$. Thus 
\begin{align*}
\nu_p(\langle e_s, e_t \rangle) &=\nu_p(a\langle e_{s'}, e_{t'} \rangle)\\
&\ge \nu_p(a)+\nu_p((x^{a-1},(x-1)^{b+1})) \\
&\ge \nu_p(a)+\nu_p((x^{a+b}))-\nu_p({{a+b}\choose {a-1}})\\
&\ge\nu_p((x^{a+b}))-\nu_p({{a+b}\choose {a}}).
\end{align*}
Now suppose there are no edges from $s_x$ to $t_x$ in $G$. Let $e_{i_1},e_{i_2},\dots,e_{i_a}$ be the edges which meet $s_{x}$, and $e_{j_1},e_{j_2},\dots,e_{j_a}$ be the edges which meet $t_{x}$. Suppose also that $e_{i_k}$ meets $t_{f(k)}$ and $e_{j_k}$ meets $s_{g(k)}$. For each $\sigma\in \mathcal{S}_a$ define a graph $G_\sigma$ as follows: 
delete the vertices $s_{x}$ and $t_{x}$ from the graph $G$ and then add edges $e'_1,\dots,e'_a$ such that $e'_k$ is incident on $s_{g(k)}$ and $t_{f(\sigma_k)}$. We shall call an admissible colouring $C\in A(G_\sigma)$ \textit{respectable} if it assigns a different colour to each of the edges $e'_1,\dots,e'_a$, and we denote the set of all such colourings $R(G_\sigma)$.

Each admissible colouring $C\in A(G)$ gives determines a $\sigma\in \mathcal{S}_a$ and a respectable colouring $C' \in R(G_\sigma)$ with $e'_1,\dots,e'_a$ having colours $c_1,\dots,c_a$ in some order as follows: the permutation $\sigma$ is chosen such that the edges $e_{i_k}$ and $e_{j(\sigma_k)}$ have the same colour, then the edges of $G_\sigma$ which appear in $G$ are given the same colour as in $C$, the edges $e'_k$ are given the same colour as $e_{i_k}$. By examining the permutations induced by the colourings we see that $$(-1)^C=(-1)^a(-1)^{C'}.$$

Conversely, a respectable colouring $C'\in R(G_\sigma)$ where $e'_1,\dots,e'_a$ have colours $c_1,\dots,c_a$ gives rise to an admissible colouring $C\in A(G)$ by giving all the edges which appear in both $G$ and $G_\sigma$ the same colour in $C$ as in $C'$, and by giving each of $e_{i_k}$ and $e_{j_k}$ the same colour as $e'_k$. Again we see that $$(-1)^C=(-1)^a(-1)^{C'},$$ and that these two operations are mutually inverse, thus
$$(-1)^{\pi_{st}}\langle e_s, e_t \rangle = (-1)^a\sum_{\sigma\in \mathcal{S}_a}\sum_C (-1)^C,$$
where the second sum is over all respectable colourings of $G_\sigma$ where the edges $e'_1,\dots,e'_a$ have colours $c_1,\dots,c_a$. There is a faithful signature preserving action of $\mathcal{S}_{a+b}$ on $R(G_\sigma)$ by permuting all the colours, so we get  
$$(-1)^{\pi_{st}}\langle e_s, e_t \rangle = (-1)^a\frac{1}{{a+b \choose a}}\sum_{\sigma\in \mathcal{S}_a}\sum_{C\in R(G_\sigma)} (-1)^C.$$

We will now show that we may replace the sum over $R(G_\sigma)$ by one over $A(G_\sigma)$.

For an admissible colouring $C\in A(G_\sigma)$ we define
$$C_l = \mid\lbrace k : e'_k \text{ has colour } c_l\rbrace \mid,$$ 
and observe that $C$ is respectable if and only if each $C_l=1$. For integers $d_1,\dots,d_{\lambda'_1}$ we define $\mathbf{C}(d_1,\dots,d_{\lambda'_1})$ to be the set of pairs $(\sigma,C')$, where $\sigma\in \mathcal{S}_{p}$ and $C'\in A(G_\sigma)$ with $C_l=d_l$ for all $l$. The group $\mathcal{S}_{d_1}\times\mathcal{S}_{d_2}\times\cdots\mathcal{S}_{d_{\lambda'_1}}$ acts, with signature, on $\mathbf{C}(d_1,\dots,d_{\lambda'_1})$ by permuting the endpoints of the edges (i.e. the elements of $\lbrace t_{f(1)},\dots, t_{f(m)}\rbrace $) of the same colour. If any of the $d_l$ are greater than one, then there is some permutation $\rho\in \mathcal{S}_{d_1}\times\mathcal{S}_{d_2}\times\cdots\mathcal{S}_{d_{\lambda'_1}}$ with negative signature, and so if $\rho C = D$ then
 $$ (-1)^D = -(-1)^C.$$
 Summing over all pairs $(\sigma,C')\in \mathbf{C}(d_1,\dots,d_{\lambda'_1})$, we obtain 
 $$ \sum_{(\sigma,C') \in\mathbf{C}(d_1,\dots,d_{\lambda'_1})} (-1)^C = -\sum_{(\sigma,C') \in\mathbf{C}(d_1,\dots,d_{\lambda'_1})} (-1)^C, $$ and so is zero, hence  $$\sum_{\sigma\in \mathcal{S}_m}\sum_{C\in A(G_\sigma)} (-1)^C  = \sum_{\sigma\in \mathcal{S}_m}\sum_{C\in R(G_\sigma)} (-1)^C.$$
Thus, $$(-1)^{\pi_{st}}\langle e_s, e_t \rangle =\frac{(-1)^a}{{a+b\choose a}}\sum_{\sigma\in \mathcal{S}_a}(-1)^{\pi_{{s_\sigma}{t_\sigma}}}\langle e_{s_\sigma}, e_{t_\sigma} \rangle,$$
and hence $\nu_p(\lambda)\ge \nu_p((x^{a+b}))-\nu_p({{a+b}\choose {a}})$ as required.
\end{proof}
The proof of \Cref{abc} is similar:
\begin{proof}[Proof of \Cref{abc}]
Suppose $\lambda=((x+1)^a,x^b,(x-1)^c)$ with $a>0$. Let $s$ and $t$ be two row equivalent $\lambda$-tableaux and construct the graph $G=G(s,t)$. If there is an edge from $s_{x+1}$ to $t_{x+1}$ then there is a faithful and signature preserving action of $\mathcal{S}_a$ on the admissible colourings of $G$, by permuting colours $c_1,\dots,c_a$. Summing the signatures of all admissible colourings of $G$ in which $e$ has colour $c_a$ we get $\frac{\langle e_s, e_t\rangle}{a}$, which is divisible by $p^{\nu_p(\langle e_s, e_t\rangle)-\nu_p(a)}$. Deleting the edge $e$ gives the graph $G(s',t')$ for $((x+1)^{a-1},x^{b+1},(x-1)^{c})$-tableaux $s'$ and $t'$. 
Exactly as before there is a signature preserving one-to-one correspondence between the admissible colourings of $G(s',t')$ and colourings of $G$ in which $e$ has colour $c_a$. Thus 
\begin{align*}
\nu_p(\langle e_s, e_t \rangle) &=\nu_p(a\langle e_{s'}, e_{t'} \rangle)\\
&\ge \nu_p(a)+\nu_p((x^{a-1},(x-1)^{b+1})) \\
&\ge \nu_p(a)+\nu_p((x^{a+b}))-\nu_p({{a+b}\choose {a-1}})-\nu_p(c!)\\
&\ge\nu_p((x^{a+b}))-\nu_p({{a+b}\choose {a}})-\nu_p(c!)
\end{align*}
as required.

Now suppose there is no edge from $s_{x+1}$ to $t_{x+1}$. Let $e_{i_1},e_{i_2},\dots,e_{i_a}$ be the edges which meet $s_{x+1}$, and $e_{j_1},e_{j_2},\dots,e_{j_a}$ be the edges which meet $t_{x+1}$. Suppose also that $e_{i_k}$ meets $t_{f(k)}$ and $e_{j_k}$ meets $s_{g(k)}$. For each $\sigma\in \mathcal{S}_a$ define a graph $G_\sigma$ as follows: 
delete the vertices $s_{\lambda_1}$ and $t_{\lambda_1}$ from the graph $G$ and then add edges $e'_1,\dots,e'_a$ and $E_1,\dots,E_c$ such that $e'_k$ is incident on $s_{g(k)}$ and $t_{f(\sigma_k)}$, and each $E_k$ goes from $s_{\lambda_1-1}$ to $t_{\lambda_1-1}$. In this context we shall call an admissible colouring $C\in A(G_\sigma)$ \textit{respectable} if it assigns a different colour to each of $e'_1,\dots,e'_a,E_1,\dots$ and $E_c$, and we denote the set of all such colourings $R(G_\sigma)$.

Each admissible colouring $C\in A(G)$ gives determines a $\sigma\in \mathcal{S}_a$ and gives rise to $c!$ respectable colourings $C' \in R(G_\sigma)$ with $e'_1,\dots,e'_a$ having colours $c_1,\dots,c_a$ in some order, while $E_1,\dots,E_c$ have the colours $c_{a+b+1},\dots,c_{a+b+c}$ in some order as follows: the permutation $\sigma$ is chosen such that the edges $e_{i_k}$ and $e_{j_(\sigma_k)}$ have the same colour, then the edges of $G_\sigma$ which appear in $G$ are given the same colour as in $C$, the edges $e'_k$ are given the same colour as $e_{i_k}$ and the edges $E_1,\dots,E_c$ are given the colours $c_{a+b+1},\dots,c_{a+b+c}$ in some order. By examining the permutations induced by the colourings we see that $$(-1)^C=(-1)^a(-1)^{C'}.$$

Conversely, a respectable colouring $C'\in R(G_\sigma)$ where $e'_1,\dots,e'_a$ have colours $c_1,\dots,c_a$ and the edges $E_1,\dots,E_c$ have the colours $c_{a+b+1},\dots,c_{a+b+c}$ gives rise to an admissible colouring $C\in A(G)$ by giving all the edges which appear in both $G$ and $G_\sigma$ the same colour in $C$ as in $C'$, and by giving each of $e_{i_k}$ and $e_{j_k}$ the same colour as $e'_k$. Again we see that $$(-1)^C=(-1)^a(-1)^{C'},$$ and we also observe that these two operations are mutually inverse, thus
$$(-1)^{\pi_{st}}\langle e_s, e_t \rangle = \frac{(-1)^a}{c!}\sum_{\sigma\in \mathcal{S}_a}\sum_C (-1)^C,$$
where the sum is over all respectable colourings of $G_\sigma$ where the edges $e'_1,\dots,e'_a$ have colours $c_1,\dots,c_a$ and $E_1,\dots,E_c$ have the colours $c_{a+b+1},\dots,c_{a+b+c}$. There is a faithful signature preserving action of $\mathcal{S}_{m}$ on $R(G_\sigma)$ by permuting all the colours, so we get  
$$(-1)^{\pi_{st}}\langle e_s, e_t \rangle = \frac{(-1)^a}{c!}\frac{1}{{a+b+c \choose a,b,c}}\sum_{\sigma\in \mathcal{S}_a}\sum_{C\in R(G_\sigma)} (-1)^C.$$
As in \Cref{base case} the sum over $R(G_\sigma)$ can be replaced by one over $A(G_\sigma)$, thus $$(-1)^{\pi_{st}}\langle e_s, e_t \rangle =\frac{(-1)^a}{c!}\frac{1}{{a+b+c \choose a,b,c}}\sum_{\sigma\in \mathcal{S}_a}(-1)^{\pi_{{s_\sigma}{t_\sigma}}}\langle e_{s_\sigma}, e_{t_\sigma} \rangle,$$ and hence $\nu_p(\lambda)\ge \nu_p((x^{a+b+c}))-\nu_p({{a+b+c}\choose {a,b,c}})-\nu_p(c!).$
\end{proof}
\newpage

\section{Schaper Numbers for $p=2$}
We shall now investigate which other partitions have high Schaper number for $p=2$. 

\begin{lemma}\label{333}
 Let $\lambda\vdash n$ and suppose there exists an $i$ such that $\lambda_i\le\lambda_{i+2}+1$ and $\lambda_{i+1}\ge 3$, then $\nu_2(\lambda)\ge 3$
\end{lemma}
\begin{proof}
By \Cref{(lambda*mu),columnremoval,abc} it suffices to show that $\nu_2((3^3))\ge 3$. We observe that this calculation has being carried out by L\"{u}beck \cite{Luebeck}, but we shall include it here for completeness. Let $s$ and $t$ be row equivalent $(3^3)$-tableau and let $G=G(s,t)$. Suppose there is a pair of edges between any two verticies; without loss of generality let these verticies be $s_1$ and $t_1$. We have already seen (\Cref{Fayers}) that $\nu_2((2^3))\ge 2$, and so, by \Cref{p edges}, we conclude that $8\mid \langle e_s, e_t \rangle$. If there are no pairs of edges then, possibly after relabelling and reordering, 

\begin{center}
 $s=
\begin{ytableau}
1&2&3\\
4&5&6\\
7&8&9\\
\end{ytableau}$ 
and
$t=
\begin{ytableau}
1&2&3\\
6&4&5\\
8&9&7\\
\end{ytableau}$,\end{center} and the polytabloids $e_s$ and $e_t$ are orthogonal.
\end{proof}

\begin{lemma}\label{4422} 
Let $\lambda\vdash n$ and suppose there exist $i$ and $j$ such that $\lambda_i=\lambda_{i+1}=\lambda_{j}+2=\lambda_{j+1}+2\ge 4$, then $\nu_2(\lambda)\ge 3$
\end{lemma}
\begin{proof}
By \Cref{(lambda*mu),columnremoval,abc} it suffices to show that $\nu_2((3^4))\ge 5$, which again has been verified by L\"{u}beck\cite{Luebeck}. It also follows from \Cref{James,p edges} by observing first $\nu_2((1^4))=3$, and then that any graph $G=G(s,t)$ where $s,t$ are $(2^4)$-tableau necessarily contains a pair of edges between two verticies which, without loss of generality, we may assume to be $s_1$ and $t_1$ and so $\nu_2((1^4))\ge 4$. Similarly any $(3^4)$-tableau necessarily contains a pair of edges between two verticies which again we may assume to be $s_1$ and $t_1$, and thus $\nu_2((3^4))\ge 5$.
\end{proof}

We are now ready to state the main results of this paper for $p=2$. 
\begin{theorem}\label{p=2k=3}
Let $\lambda \vdash n$ and $p=2$. Then $\nu_p(\lambda)\ge 3$ if and only if one of the following hold:
\begin{enumerate}
\item $\lambda$ is triply $2$-singular.
\item There exist $i,j$ with $\{i, i+1, i+2\}\cap\{j,j+1\}=\emptyset$ such that $\lambda_i\le\lambda_{i+2}+1$ and $\lambda_{i+1}\ge 2$ and $\lambda_j=\lambda_{j+1}$.
\item $\lambda$ is 4-singular.
\item There exist $i,j$ such that $\lambda_i=\lambda_{i+1}=\lambda_{j}+2=\lambda_{j+1}+2\ge 4$.
\item There exist $i$ such that $\lambda_i\le\lambda_{i+2}+1$ and $\lambda_{i+1}\ge 3$.
\end{enumerate}
\end{theorem}
\begin{proof}
The `if' direction is \Cref{firstconditions}, \Cref{James} and \Cref{333,4422}. 
To prove the `only if' direction we must show that if $\lambda$ satisfies one of the properties of \Cref{Fayers} but none of the properties above then $\nu_2(\lambda) = 2$.
First, suppose $\lambda$ is doubly $2$-singular and let $\lambda_i=\lambda_{i+1}$ and $\lambda_j=\lambda_{j+1}$ be the two disjoint singularities. As $\lambda$ is not 4-singular and does not satisfy \emph{4.} or \emph{5.} above, we may assume that $\lambda_i\ge \lambda_j+3$ and also that there are no other rows of length $\lambda_i$ or $\lambda_j$, nor are there rows of lengths $\lambda_i \pm 1 $ or $\lambda_j \pm 1$. In this case $\lambda^r$, the 2-regularisation, \cite{JamesKerber}, of $\lambda$, is 
$${\lambda^r}_{k}=
\begin{cases}
\lambda_k & k\notin \{i,i+1,j,j+1\}\\
\lambda_k+1 & k \in \{i,j\}\\
\lambda_k-1 & k\in \{i+1,j+1\}
\end{cases}.
$$
We shall show that $D^{\lambda^r}$ is in the second Schaper layer, and thus the Schaper number of $\lambda$ is two. As $[S^\lambda:D^{\lambda^r}]=1$, the value of $\sum_{i=1} [S^\lambda_{(i)}:D^{\lambda^r}]$ is the number of the layer in which $D^{\lambda^r}$ appears. By Schaper's formula, $$\sum_{i=1} [S^\lambda_{(i)}:D^{\lambda^r}]=\sum_{\nu} a_\nu [S^\nu:D^{\lambda^r}],$$ for $\nu \rhd \lambda$. As $[S^\nu:D^{\lambda^r}]=0$ for all  $\nu \rhd \lambda^r$, the sum is over all $\nu$ such that $\lambda\lhd\nu\unlhd\lambda^r$, and thus any $\nu$ contributing to the sum must have $\nu_k=\lambda_k$ for all $k\notin \{i,i+1,j,j+1\}$. Also, $a_\nu$ is zero unless there is are rim-hooks $g$ and $h$ of $Y(\lambda)$ and $Y(\nu)$ respectively such that $\nu_p(\mid g \mid) \ne 0$ and $Y(\lambda\backslash g)=Y(\nu \backslash h)$. The only contributing terms are when $\nu\in\{\lambda',\lambda''\}$ where 

$\lambda'_k=\begin{cases}
\lambda_k & k\notin\{i,i+1\}\\
\lambda_k+1 & k=i\\
\lambda_k-1 & k=i+1
\end{cases}$
and  
$\lambda''_k=\begin{cases}
\lambda_k & k\notin\{j,j+1\}\\
\lambda_k+1 & k=j\\
\lambda_k-1 & k=j+1
\end{cases}$,\newline with $a_{\lambda'}=a_{\lambda''}=1$.
By row and column removal \cite{Jamesrowremoval}, or by observing that each of these partitions have $\lambda^r$ as their 2-regularisations, we see that $[S^\nu:D^{\lambda^r}]=1$ for $\nu\in\{\lambda',\lambda''\}$, and thus $\sum_{i=1} [S^\lambda_{(i)}:D^{\lambda^r}]=2$ as required. 
If $\lambda$ satisfies property \emph{2.} of \Cref{Fayers}, but none of the above, the only 2-sigularity in $\lambda$ is a pair of rows of length 2 and we conclude $\nu_2(\lambda)=2$ by \Cref{James}.
\end{proof}
As before \Cref{(lambda*mu)} allows us to get some conditions for which  $\nu_2(\lambda)\ge 4$. These are the first six conditions below.
\begin{theorem}\label{p=2k=4}
Let $\lambda \vdash n$ and $p=2$. Then $\nu_p(\lambda)\ge 4$ if and only if one of the following hold:
\begin{enumerate}
    \item $\lambda$ is quadruply $2$-singular; that is there are $i,j,k$ and $l$ such that $\lambda_i=\lambda_{i+1}$, $\lambda_j=\lambda_{j+1}$, $\lambda_k=\lambda_{k+1}$ and $\lambda_l=\lambda_{l+1}$ with $\{i, i+1\}$ $\{j, j+1\}$, $\{k, k+1\}$ and $\{l, l+1\}$ pairwise disjoint.
    \item There exists $i$ such that $\lambda_i=\lambda_{i+3}=1$ and $j\notin \{i,\dots,i+3\}$ with $\lambda_j=\lambda_{j+1}$.
	\item There exist $i,j$ with $\{i,\dots, i+2p-2\}\cap\{j,\dots,j+p-1\}=\emptyset$ such that $\lambda_i\le\lambda_{i+2p-2}+1$ and $\lambda_{i+p-1}\ge 3$ and $\lambda_j=\lambda_{j+p-1}$.
	\item There exists $i,j,k$ with $i \ge j+p\ge k+p$ and $\lambda_i=\lambda_{i+p-1}$ and $\lambda_j=\lambda_{j+p-1}$  and $\lambda_k\le\lambda_{i+2p-2}+1$ and $\lambda_{i+p-1} = 2$.
	\item There exist $i,j$ with $\{i,i+1,i+2\}\cap\{j,j+1\}=\emptyset$ such that $\lambda_i\le\lambda_{i+2}+1$ and $\lambda_{i+1}\ge 3$ and $\lambda_j=\lambda_{j+1}$.
	\item There exist $i,j,k$ with $\{i,i+1\}\cap\{j,j+1\}\cap\{k,k+1\}=\emptyset$ such that $\lambda_i=\lambda_{i+1}=\lambda_{j}+2=\lambda{j+1}+2\ge 4$ and $\lambda_k=\lambda_{k+1}$.
	\item There exists $i$ such that $\lambda_i=\lambda_{i+3}>1$.
	\item There exists $i$ such that $\lambda_i\le\lambda_{i+3}+2$ with $\lambda_{i+1}\ge 4$, $\lambda_{i+2}\ge 3$ and $\lambda_{i+3}\ge 1$.
\end{enumerate}
\end{theorem}
\begin{proof}
Observe that the `if' direction follows from \Cref{p=2k=3}, \Cref{Fayers}, \Cref{(lambda*mu)} and \Cref{James} for conditions 1-6. We observed that $\nu_2((2^4))\ge 4$ in the proof of \Cref{4422} above. Also in that proof we show that $\nu_2((3^4))\ge 5$ and hence $\nu_2(\lambda)\ge 4$ for $\lambda\in\{(4,4,3,2),(4,4,4,2),(4,4,3,3),(4,4,4,3),(5,4,4,3)\}$. To see that a partition satisfying \emph{8.} has $\nu_2(\lambda)\ge 4$ it remains to check this for $\lambda\in\{(5,4,4,4),(6,5,4,4)\}$. This follows from the fact that $\nu_2((4^4))\ge 6$, which can be checked by computing the inner products of polytabloids $e_s$ and $e_t$ for all $s$ and $t$ where $G(s,t)$ contains no pairs of edges.

To prove the `only if' direction we will show that if $\lambda$ satisfies one of the conditions from \Cref{p=2k=3}, but none of the above conditions, then the Schaper number of $\lambda$ is three. If $\lambda$ is triply $p$-singular, with $\lambda_i=\lambda_{i+1}$, $\lambda_j=\lambda_{j+1}$ and  $\lambda_k=\lambda_{k+1}$, then similarly to before these lengths all differ by at least 3 and all other rows have lengths that differ by at least 2 from $\lambda_i,\lambda_j$ and $\lambda_k$. The only contributing terms in the sum $\sum_{\nu} a_\nu [S^\nu:D^{\lambda^r}]$ are those for which $\nu\in\{\lambda'\lambda'',\lambda^{(3)}\}$, where 
$\lambda'_l=\begin{cases}
\lambda_l & l\notin\{i,i+1\}\\
\lambda_l+1 & l=i\\
\lambda_l-1 & l=i+1
\end{cases},$

$\lambda''_l=\begin{cases}
\lambda_l & l\notin\{j,j+1\}\\
\lambda_l+1 & l=j\\
\lambda_l-1 & l=j+1
\end{cases}$
and
$\lambda^{(3)}_l=\begin{cases}
\lambda_l & l\notin\{k,k+1\}\\
\lambda_l+1 & l=k\\
\lambda_l-1 & k=j+1
\end{cases},$
\newline which all appear with coefficient $a_{\lambda'}=a_{\lambda''}=1$.
As before $[S^\nu:D^{\lambda^r}]=1$ if $\nu$ is one of the above, as the $p$-regularisation of each of these $\nu$ is $D^{\lambda^r}$, and thus $\nu_2(\lambda)=3$.

If $\lambda$ is 4-singular, but does not satisfy any of the conditions above, then the rows of the same length are of length 1 and $\lambda$ is not 6-singular so, by \Cref{James}, $\nu_2(\lambda)=3$. 

Let $\lambda$ satisfy property \emph{4.} of \Cref{p=2k=3} but none of the above. If there are two rows of length 3, then by \Cref{James}, $\nu_2(\lambda)=3$, so we may assume $\lambda\in \{(\cdots,k+2,k,k,k-1,k-3,\cdots),(\cdots,k+3,k+1,k,k,k-2,\cdots),(\cdots,k+3,k,k,k,k-3,\cdots)\}$ with $k\ge 4$. In all three cases, just as before, we shall show that the simple module corresponding to the $p$-regularisation of $\lambda$ lies in the 3rd, and therefore top, Schaper layer. 

Let $\lambda = (\cdots,k+2,k,k,k-1,k-3,\cdots)$, then $\lambda^r=(\cdots,k+2,k+1,k,k-2,k-3,\cdots)$. The only $\nu$ contributing to the sum $\sum_{\nu} a_\nu [S^\nu:D^{\lambda^r}]$ are $(\cdots,k+2,k+1,k-1,k-1,k-3,\cdots)$ which appears with coefficient 1, and  $\lambda^r$ itself, which appears with coefficient 2. Both of these have $[S^\nu:D^{\lambda^r}]=1$, as the $p$-regularisation of both $\nu$ and $\lambda^r$ is $\lambda^r$, and hence $\sum_{\nu} a_\nu [S^\nu:D^{\lambda^r}]=3=\nu_2(\lambda)$

Now consider $\lambda = (\cdots,k+3,k,k,k,k-3,\cdots)$, then $\lambda^r=(\cdots,k+3,k+2,k,k-2,k-3,\cdots)$. The only $\nu$ contributing to the sum $\sum_{\nu} a_\nu [S^\nu:D^{\lambda^r}]$ are $(\cdots,k+3,k+2,k-1,k-1,k-3,\cdots)$, $(\cdots,k+3,k+1,k+1,k-2,k-3,\cdots)$ and $\lambda^r$ itself, which all appear with coefficient 1 and have $[S^\nu:D^{\lambda^r}]=1$, as before, so $\sum_{\nu} a_\nu [S^\nu:D^{\lambda^r}]=3=\nu_2(\lambda)$.

Finally, if $\lambda = (\cdots,k+3,k+1,k,k,k-2,\cdots)$, then $\lambda^r=(\cdots,k+3,k+2,k,k-1,k-3,\cdots)$. The only $\nu$ contributing are $(\cdots,k+3,k+1,k+1,k-1,k-3,\cdots)$, with coefficient 1, and $\lambda^r$ itself, with coefficient 2. Again both have $[S^\nu:D^{\lambda^r}]=1$ so $\sum_{\nu} a_\nu [S^\nu:D^{\lambda^r}]=3=\nu_2(\lambda)$.

If $\lambda$ satisfies property \emph{2.} of \Cref{p=2k=3} but none of the above then $\lambda_i=2$ and $\lambda_j\notin \{2,3\}$. If $\lambda_j=1$ then $\nu_2(\lambda)=3$ by \Cref{James}, so we may assume $\lambda=(r,r-1,\dots,m+2,m,m,m-2,\dots,3,2,2,1)$. Suppose further that $m\ne 4$.

We shall construct row equivalent $\lambda$-tableaux $t$ and $u$ such that $16\nmid \langle e_t,e_u\rangle$. We shall choose $t$ to be the initial tableaux, that is the tableaux whose entries are, from left to right and top to bottom, $1,2,3,\dots$. We then choose $u$ to be the unique tableaux which is row equivalent to $t$ and whose rows of unique length have entries in descending order from left to right and whose rows of length $m$ are obtained from $t$ by permuting the other rows that occur as a pair as described below: 
If the pair of rows of length $m$ appearing in $t$ is
$$\ytableausetup{boxsize=2.4em}
\begin{ytableau}
a_1&a_2&a_3&a_4&a_5&a_6&a_7&a_8&\dots& a_{m-1} &a_{m}\\
b_1&b_2&b_3&b_4&b_5&b_6&b_7&b_8&\dots& b_{m-1} &b_{m}\\
\end{ytableau},$$ 

then set the corresponding rows of $u$ to be 

$$\ytableausetup{boxsize=2.4em}
\begin{ytableau}
a_{m}& a_{m-2} &a_{m-1}&\dots&a_8&a_5&a_6&a_3&a_4&a_1&a_2\\
b_{m-1}&b_m&b_{m-3}&\dots&b_6&b_7&b_4&b_5& b_2 &b_3& b_1\\
\end{ytableau},$$

and set the last rows of $u$ to be $$\ytableausetup{boxsize=2.4em}
\begin{ytableau}
n-3&n-4\\
n-2&n-1\\
n
\end{ytableau}.$$  

It is easy to see that any tabloid $\{v\}$ common to $e_t$ and $e_u$ must have $R_i(\{v\})=R_i(t)$ for any row $i$ of unique length with $\mid R_i(t)\mid \ne 1$. For example, the elements occurring first in each row of $t$ occur last in the rows in $u$ except the row of length $m$ where it is the second to last entry. Apart from in this row, these entries can not appear lower in $\{v\}$ than they do in $t$ and so they must appear in the same row. Similarly we see that if $\lambda_l=\lambda_{l+1}$ then $R_l(\{v\})\cup R_{l+1}(\{v\})=R_l(t)\cup R_{l+1}(t)$ and thus  $\langle e_{t},e_{u}\rangle=\langle e_{t'},e_{u'}\rangle\cdot \langle e_{t''},e_{u''}\rangle$, where $t'$ is the tableau consisting of only the pair of rows in $t$ of length $m$ and $t''$ is the tableau consisting of last three rows of $t$, with $u'$ and $u''$ defined similarly. It is easy to see that $\langle e_{t'},e_{u'}\rangle = 2$ and $ \langle e_{t''},e_{u''}\rangle=12$ and thus $\langle e_{t},e_{u}\rangle=24$, which is not divisible by 16.

We may do a similar thing if $m=4$, in which case we may assume $\lambda=(\cdots,7,6,4,4,2,2,1)$. If we set $t$ to be the initial tableau and set $u$ to be the row equivalent tableau with entries in descending order in all rows except the rows of length 4 which we set to 
$$\ytableausetup{boxsize=2em}
\begin{ytableau}
a_{3}& a_{4} &a_{1}&a_2\\
b_4& b_2 &b_3& b_1\\
\end{ytableau},$$ as before. In this case we see that $\langle e_{t},e_{u}\rangle=\langle e_{t'},e_{u'}\rangle$, where 
$$
t'=\ytableausetup{boxsize=2em}
\begin{ytableau}
1&2&3&4\\
5&6&7&8\\
9&10\\
11&12\\
13
\end{ytableau}$$
and 
$$
u'=\ytableausetup{boxsize=2em}
\begin{ytableau}
3&4&1&2\\
8&6&7&5\\
10&9\\
12&11\\
13
\end{ytableau}.$$ This inner product is 8, and hence not divisible by 16, so the Schaper number of $\lambda$ is at most three.

Now suppose $\lambda$ satisfies the final property of \Cref{p=2k=3}. Recall the case where $\lambda$ has two rows of length 2 and one of length 1 was dealt with earlier, so we may assume $\lambda = (\cdots,k+3,k+1,k+1,k-1,k-1,k-3,\cdots)$ and thus $\lambda^r = (\cdots,k+3,k+2,k+1,k-1,k-2,k-3,\cdots)$. The contributing terms are $(\cdots,k+3,k+2,k,k-1,k-1,k-3,\cdots)$, $(\cdots,k+3,k+1,k+1,k,k-2,k-3,\cdots)$ and $\lambda^r$, all with coefficient 1. All of these have $\lambda^r$ as their $p$-regularisation so $[S^\nu:D^{\lambda^r}]=1$ and therefore $\sum_{\nu} a_\nu [S^\nu:D^{\lambda^r}]=3=\nu_2(\lambda)$, completing the proof.
\end{proof}
\newpage

\section{Schaper Numbers for Odd Primes}
The problem of characterising partitions with high Schaper number for odd primes is more difficult. Unlike in \Cref{Fayers}, where there is a nice characterisation for all primes, small primes must be treated separately when characterising partitions with higher Schaper numbers. In this section we will give a necessary list of conditions for partitions to have Schaper number at least three for odd primes. Throughout this section $p$ is assumed to be odd.

\begin{theorem}\label{onlyif}
Let $\lambda \vdash n$ and $p$ be an odd prime, then $\nu_p(\lambda)\ge 3$ only if one of the following conditions hold:
\begin{enumerate}
\item $\lambda$ is triply $p$-singular.
\item There exist $i,j$ with $\{i,\dots, i+2p-2\}\cap\{j,\dots,j+p-1\}=\emptyset$ such that $\lambda_i\le\lambda_{i+2p-2}+1$ and $\lambda_{i+p-1}\ge 2$ and $\lambda_j=\lambda_{j+p-1}$.
\item There exists $i$ with $\lambda_{i}=\lambda_{i+2p-1}\ge 2$.
\item There exists $i$ such that $\lambda_i=\lambda_{i+p-1}=\lambda_{i+p}+1=\lambda_{i+2p-1}+1\ge 3$.
\item There exist $i$ such that $\lambda_i\le\lambda_{i+3p-3}+2$ with $\lambda_{i+2p-2}\ge 2$ and some $p$ consecutive rows between $\lambda_i$ and $\lambda_{i+3p-3}$ have length $k\ge 3$.
\end{enumerate}
\end{theorem}
\begin{proof}

We shall show that if $\lambda$ satisfies one of the conditions of \Cref{Fayers}, but not any of the above then $\nu_p(\lambda)=2$. 

First suppose $\lambda$ is doubly $p$-singular. If the $p$-singularities are of the same length then this length must be 1, and so we are done by \Cref{James}. If they differ in length by 1 then $\lambda=(\cdots,4^{i_4},3^{i_3},2^{p+i_2},1^{p+i_1})$  with $i_j\le p-1$ for all $j$, $i_1+i_2\le p-2$ and $i_2+i_3\le p-2$ . In this case we see that $p^3\nmid \langle e_t,e_u\rangle$ where $t$ is the initial $\lambda$-tableau and $u$ is the tableau obtained from $t$ by reversing the entries in all rows except $i_2+1$ of the rows of length 2.

Suppose the lengths of these two singularities differ by 2 or more and that neither of them are of length 1, We will now show that the module $D^{\lambda^r}$ appears in the second Schaper layer of $S^\lambda$. Observe that the when we take the $p$-regularisation of such a partition boxes can only move into the next position in the $p$ ladder; that is to say a box is either fixed or it moves up $p$ rows and into the column to its right. This is because if it were able to move further then we must have $2p-1$ rows who differ by 2, or $3p-2$ rows who differ by 3. 

Again, $[S^\lambda:D^{\lambda^r}]=1$ so the value of $\sum_{i=1} [S^\lambda_{(i)}:D^{\lambda^r}]=\sum_{\nu} a_\nu [S^\nu:D^{\lambda^r}]$ is the number of the Schaper layer in which $D^{\lambda^r}$ appears. The term $[S^\nu:D^{\lambda^r}]$ can only contribute if $\nu$ is obtained from $\lambda$ by unwrapping a single $mp$-hook and wrapping it further up the Young diagram, and if $\lambda \lhd \nu \unlhd \lambda^r$. 

Any hook which contains boxes not in one of the two singularities would result in a $\nu$ which is not dominated by $\lambda^r$ so the only options are the two $p$-hooks which have their foot in the removable box of a $p$-singularity. Such a hook must then be wrapped in a way so that all of its boxes are placed in the same column that they appear in $[\lambda]$, or the column immediately to the right. The leg length of the hook as it appears in $[\lambda]$ is $p$ and in $[\lambda^r]$ t is $p-1$, so the coefficient $a_\nu=+1$. Also, as $\nu^r=\lambda^r$ we have $[S^\nu:D^{\lambda^r}]=1$, and hence $\sum_{i=1} [S^\lambda_{(i)}:D^{\lambda^r}]=\sum_{\nu} a_\nu [S^\nu:D^{\lambda^r}]=2$, as claimed.

Now suppose that the lengths of these two singularities differ by 2 or more and that there is a  $p$-singularity of length 1. We shall construct $\lambda$-tableaux $t$ and $u$ such that the inner product between the polytabloids $e_t$ and $e_u$ is divisible by $p^2$ but not $p^3$. We may assume that $\lambda=(\cdots,k+1^{i_{k+1}},k^{p+i_k},k-1^{i_{k-1}},\cdots,2^{i_2},1^{p+i_1})$, with $i_{k+1}+i_k+i_{k-1} < p-2$ and $i_j < p$ for all $j$. As before we choose $t$ to be the initial $\lambda$-tableau and $u$ to be the tableau row equivalent to $t$ which is obtained by reversing the order of entries in all of the rows except for $p+i_k-\text{max}\{i_{k+1},i_{k-1}\}$ of the rows of length $k$. Of these remaining rows, we set $p-\text{max}\{i_{k+1},i_{k-1}\}-1$ of these to 
$$\begin{ytableau}
\cdots & a_5 & a_6 & a_3& a_4 &a_1&a_2\\
\end{ytableau}$$
and the other $i_{k}+1$ rows to 
$$\begin{ytableau}
\cdots &a_7 & a_4 & a_5 & a_2& a_3 &a_1\\
\end{ytableau},$$
where 
$$\begin{ytableau}
a_1 & a_2 & a_3 & a_4&\cdots & \cdots & a_k\\
\end{ytableau}$$
is the corresponding row of $t$.
 
First observe that any entry that appears in a row of length $i$ for $i\notin \{k-1,k,k+1\}$ of a tabloid common to $e_t$ and $e_u$ must also appear in a row of that length in $t$ and $u$. This allows us to deduce that $\langle e_t, e_u\rangle = \langle e_{t'},e_{u'}\rangle \langle e_{t''}, e_{u''}\rangle$ where $t'$ and $u'$ are the tableau whose rows are the same as the rows of $t$ and $u$ whose length is not $k-1$, $k$ or $k+1$, and $t''$ and $u''$ are the $(k+1^i_{k+1},k^{p+i_{k}},k-1^i_{k-1})$-tableaux whose rows are the same as the corresponding rows of $t$ and $u$ respectively.
Observe also that $\nu_p( \langle e_{t'},e_{u'}\rangle)=1$ so to complete the proof it remains to prove that $p^2\nmid \langle e_{t''},e_{u''}\rangle$.

To see this consider the tableaux 
$$t''=\ytableausetup{boxsize=2em}
\begin{ytableau}
\Tilde{a_1}&\Tilde{a_2}&\Tilde{a_3}&\cdots&\cdots&\cdots & \cdots &\Tilde{a_{k+1}}\\
\Tilde{x_1}&\Tilde{x_2}&\Tilde{x_3}&\cdots&\cdots&\cdots &\Tilde{x_{k}}\\
\Tilde{y_1}&\Tilde{y_2}&\Tilde{y_3}&\cdots&\cdots&\cdots &\Tilde{y_{k}}\\
\Tilde{z_1}&\Tilde{z_2}&\Tilde{z_3}&\cdots&\cdots &\cdots&\Tilde{z_{k}}\\
\Tilde{c_1}&\Tilde{c_2}&\Tilde{c_3}&\cdots&\cdots&\Tilde{c_{k-1}}\\
\end{ytableau}$$
and 
$$u''=
\ytableausetup{boxsize=2em}\begin{ytableau}
\Tilde{a_{k+1}}&\Tilde{a_k}&\cdots&\cdots&\cdots&\cdots & \Tilde{a_2} &\Tilde{a_1}\\
\cdots&\cdots&\cdots &\Tilde{x_3}&\Tilde{x_4}&\Tilde{x_1}&\Tilde{x_{2}}\\
\cdots&\cdots&\cdots &\Tilde{y_5}&\Tilde{y_2}&\Tilde{y_3}&\Tilde{y_{1}}\\
\Tilde{z_{k}}&\cdots&\cdots&\Tilde{z_4}&\Tilde{z_3}&\Tilde{z_2}&\Tilde{z_1}\\
\Tilde{c_{k-1}}&\cdots&\cdots &\Tilde{c_3}&\Tilde{c_2}&\Tilde{c_1}\\
\end{ytableau},$$
where the $\Tilde{a_i}, \Tilde{x_i}, \Tilde{y_i}, \Tilde{z_i}$ and $\Tilde{c_i}$ are represent columns of length $i_{k+1}, p+i_k-\text{max}\{i_{k+1},i_{k-1}\}, i_k+1, \text{max}\{i_{k+1},i_{k-1}\}$ and $i_{k-1}$ respectively. Observe that for any tabloid $\{T\}$ common to $e_{t''}$ and $e_{u''}$, the permutations required to make $t''$ and $u''$ row equivalent to $T$ have the same number of transpositions and therefore the same sign. This means that $\langle e_{t''}, e_{u''}\rangle$ is the number of tabloids $\{T\}$ common to $e_{t''}$ and $e_{u''}$. We shall count such tabloids by constructing tableau $U$ which are column equivalent to $u''$ and row equivalent to $T$. Observe that once we have chosen which $p-\text{max}\{i_{k+1},i_{k-1}\}$ of the rows of length $k$ in $U$ have entries in their last box which come from the second column of $t''$ (of which there are ${p+i_k \choose p-\text{max}\{i_{k+1},i_{k-1}\}}$ possible choices, a number divisible by $p$) then $U$ is chosen by choosing the order in which entries in the other columns appear. By considering that $U$ must be row equivalent to some tableau which is column equivalent to $t''$ we observe that we are only choosing the order of either $i_{k+1}+p-\text{max}\{i_{k+1},i_{k-1}\}-1$, $i_{k+1}+i_{k}+1$, $p-\text{max}\{i_{k+1},i_{k-1}\}-1+\text{max}\{i_{k+1},i_{k-1}\}$, $p-\text{max}\{i_{k+1},i_{k-1}\}-1+i_{k-1}$,$i_{k}+1+\text{max}\{i_{k+1},i_{k-1}\}$ or $i_{k}+1+i_{k-1}$ elements. The number of possible choices here is the product of the factorials of these numbers, which is not divisible by $p$. We conclude that $p^2\nmid \langle e_{t''},e_{u''}\rangle$ and thus $\nu_p ( \langle e_{t},e_{u}\rangle) = 2$, as required. 

Now suppose that $\lambda$ satisfies the other condition of \Cref{Fayers}, but not any of the above, that is there exist $i$ such that $\lambda_i\le\lambda_{i+2p-2}+1$ and $\lambda_{i+p-1}\ge 2$. As before we shall show that the $D^{\lambda^r}$ appears in the second layer, as in the proof of \Cref{p=2k=3}. 

We may assume that $$\lambda=(\cdots,k+1^{i_{k+1}},k^{p+i_k},k-1^{i_{k-1}},k-2^{i_{k-2}},\cdots,2^{i_2},1^{i_1}),$$ 

with $i_{k+1}+i_{k} < p-1$, $i_{k}+i_{k-1}\ge p-1$, and not satisfying any of the conditions of \Cref{onlyif}, or that 
$$\lambda=(\cdots,k+1^{i_{k+1}},k^{i_k},k-1^{p+i_{k-1}},k-2^{i_{k-2}},\cdots,2^{i_2},1^{i_1})$$
with $i_{k}+i_{k-1}\ge p-1$, and not satisfying any of the conditions of \Cref{onlyif}. In the first case the $p$-regularisation of $\lambda$ is 
\begin{align*}
\lambda^r=(\cdots &,k+1^{i_{k+1}+i_k+1},k^{i_{k-1}}, k-1^{2p-i_{k}-i_{k+1}-3},\\
&  \quad k-2^{i_{k-2}+i_{k+1}+i_k+2-p},k-3^{i_{k-3}},\cdots,2^{i_2},1^{i_1}),\\
\end{align*}
  while in the second it is 
\begin{align*}
\lambda^r=(\cdots &,k+1^{i_{k+1}+i_k+i_{k-1}+2-p},k^{2p-i_k-i_{k-1}-3},\\
&k-1^{i_k+i_{k-1}+i_{k-2}+2-p},k-2^{2p-3-i_{k-1}-i_{k-2}},k-3^{i_{k-1}+i_{k-2}+i_{k-3}},\cdots),
\end{align*}
 if $i_{k-1}+i_{k-2}\ge p-1$ and 
\begin{align*}
\lambda^r=(\cdots,&k+1^{i_{k+1}+i_k+i_{k-1}+2-p},k^{2p-i_{k}-i_{k-1}-3},\\
&k-1^{i_{k-1}},
k-2^{i_{k-1}+i_{k-2}+1},k-3^{i_{k-3}},\cdots)
\end{align*} otherwise.
Observe that in each of these cases the only $\mu$ that can contribute to the sum in \Cref{Schaper's sum formula} are those $\mu$ which are obtained from $\lambda$ by unwrapping an $mp$ hook and wrapping it back on higher up the diagram in such a way that $\lambda\lhd \mu \unlhd \lambda^r$. Observe that there are only two such $mp$ hooks. One is the $p$-hook whose foot is in the row of the same length as the $p$-singularity, and the other is a $2p$ hook. There is a unique way that each of these can be wrapped and each of these has $p$-regularisation $\lambda^r$, hence each will contribute one to the sum, and thus $\sum_{i=1} [S^\lambda_{(i)}:D^{\lambda^r}]=2$ and $D^{\lambda^r}$ appears in the second layer.
\end{proof}

We shall now investigate which of these conditions are sufficient for $\nu_p(\lambda)\ge 3$, for which we make the following conjecture:

\begin{conj}
Let $\lambda\vdash n$ satisfy one of the conditions of \Cref{onlyif}, then $\nu_p(\lambda)\ge 3$. Thus the `only if' in \Cref{onlyif} can be replaced by `if and only if' giving a complete characterisation of partitions with Schaper number at least three. 
\end{conj}
\begin{lemma}\label{2^2p}
Let $\lambda\vdash n$ and suppose there exists an $i$ with $\lambda_{i}=\lambda_{i+2p-1}\ge 2$, then $\nu_p(\lambda)\ge 3$. 
\end{lemma}
\begin{proof}
By \Cref{(lambda*mu),columnremoval} is suffices to show $\nu_p(\lambda)\ge 3$ for $\lambda=(2^{2p})$. Let $s$ and $t$ be row equivalent $(2^{2p})$-tableaux and let $G=G(s,t)$ as above. Observe that, without loss of generality, there are $a \ge p$ edges from $s_1$ to $t_1$, so $p^3\mid \langle e_s, e_t\rangle$ by \Cref{p edges}
\end{proof}
\begin{lemma}\label{3^p2^p}
Let $\lambda\vdash n$ and suppose there exists $i$ such that $\lambda_i=\lambda_{i+p-1}=\lambda_{i+p}+1=\lambda_{i+2p-1}+1\ge 3$. Then $\nu_p(\lambda)\ge 3$.
\end{lemma}
\begin{proof}
Again, by \Cref{(lambda*mu),columnremoval} we are reduced to showing $\nu_p((3^p,2^p))\ge 3$.
Let $\lambda=(3^p,2^p)$ let $s$ and $t$ be row equivalent $\lambda$-tableaux. Consider the graph $G=G(s,t)$. If this graph contains no edges from $s_3$ to $t_3$ then by deleting these two vertices we obtain the graph $G_\sigma$ for some $s_\sigma,t_\sigma$ row equivalent $(2^{2p})$-tableaux. There is a one-to-one correspondence between admissible colourings $C$ of $G$ and pairs $(\sigma,C')$ where $\sigma\in \mathcal{S}_p$ and $C'$ is an admissible colouring of $G_\sigma$ where the edges $e_1',\dots,e_p'$ have colours $c_1,\dots,c_p$ in some order. Examining permutations induced by the colourings shows 
$$(-1)^C = (-1)^p(-1)^{C'} ,$$
thus
$$ (-1)^{\pi_{st}}\langle e_s, e_t\rangle = (-1)^p\sum_{\sigma\in\mathcal{S}_r}\sum_{C}(-1)^C,$$ where the second sum is over the admissible colourings of $G_\sigma$ where the edges $e_1',\dots,e_p'$ have colours $c_1,\dots,c_p$.
We define the set of respectable colourings of $G_\sigma$, denoted $R(G_\sigma)$, to be the set of colourings of $G_\sigma$ in which the edges $e_i'$ all have different colours. There is a faithful signature preserving action of $\mathcal{S}_{2p}$ on $R(G_\sigma)$ given by permuting the colours $c_1,\dots,c_{2p}$, so 
$$ (-1)^{\pi_{st}}\langle e_s, e_t\rangle =\frac{1}{{2p \choose p}} (-1)^p\sum_{\sigma\in\mathcal{S}_r}\sum_{C\in R(g_\sigma)}(-1)^C.$$ 
As in the proof of \Cref{base case} we may replace the sum over $R(G_\sigma)$ by one over $A(G_\sigma)$, thus $$ (-1)^{\pi_{st}}\langle e_s, e_t\rangle =\frac{1}{{2p \choose p}} (-1)^p\sum_{\sigma\in\mathcal{S}_r}\sum_{C\in A(G_\sigma)} (-1)^{\pi_{s_\sigma t_\sigma}}\langle e_{s_\sigma}, e_{t_\sigma}\rangle .$$ 
As the binomial coefficient ${2p \choose p}$ is not divisible by $p$ and as $p^3$ divides $\langle e_{s_\sigma}, e_{t_\sigma}\rangle,$ by \Cref{2^2p} we conclude that $p^3$ divides $\langle e_s, e_t\rangle.$ 
 
On the other hand, if there is an edge from $s_3$ to $t_3$, then there is a faithful and signature preserving action of $\mathcal{S}_p$ on the admissible colourings of $G$, by permuting the colours $c_1,\dots,c_p$. Summing the signatures of all admissible colourings of $G$ in which $e$ has colour $c_p$ we get $\frac{\langle e_s, e_t\rangle}{p}$, which is divisible by $p^2$ if and only if $\langle e_s, e_t\rangle$ is divisible by $p^3$. Deleting the edge $e$ gives the graph $G(s',t')$ for $ (3^{p-1},2^{p+1})$-tableaux $s'$ and $t'$. 
 There is a one-to-one correspondence between the admissible colourings of $G(s',t')$ and colourings of $G$ in which $e$ has colour $c_p$. 
 This correspondence is signature preserving and so the sum 
 of 
 $(-1)^C$ 
 over all admissible colourings $C$ of $G$ 
 in which $e$ has colour $c_p$ is $\langle e_{s'}, e_{t'}\rangle=\frac{\langle e_s, e_t\rangle}{p}$ which is divisible by $p^2$, by \Cref{Fayers}.
\end{proof} 

\begin{lemma}\label{321432}
Let $\lambda\vdash n$ and suppose there exist $i$ such that $\lambda_i\le\lambda_{i+3p-3}+2$ with $\lambda_{i+2p-2}\ge 2$. Suppose further that the $p$-singularity between rows $i$ and $i+3p-2$ has length $\lambda_j\neq\lambda_i-2$ and $\lambda_j\ge 3$. 
\end{lemma} 
\begin{proof}
By \Cref{(lambda*mu),columnremoval} it suffices show the following partitions have Schaper number greater than or equal to three:
\begin{itemize}
    \item $(3^a,2^b,1^c)$ with $a\ge p$, $a+b\ge 2p-1$ and $a+b+c=3p-2$,
    \item $(4^a,3^b,2^c)$ with $a+b+c=3p-2$, and one of $a$ or $b\ge p$ 
\end{itemize}
which, by \Cref{abc} follows from: 
\begin{claim}
$\nu_p((3^{3p-2}))=4$.
\end{claim}
Let $s$ and $t$ be row equivalent $(3^{3p-2})$-tableaux. We may assume there are at least $p$ edges from $s_1$ to $t_1$, and as $\nu_p((2^{3p-2}))\ge 3$, we deduce $\nu_p((3^{3p-2}))=4$ by \Cref{p edges}, thus proving the claim.
\end{proof}

\begin{remark}
Observe that \Cref{(lambda*mu),Fayers,2^2p,3^p2^p,321432} show that all the conditions in \Cref{onlyif} are sufficient, except possibly the last, for which only the case $\lambda=(5^a,4^b,3^c)$ where $a+b+c=3p-2$ and $a,b<p$ remains. To prove the conjecture it only remains to show that such a partition also has $\nu_p(\lambda)\ge 3$.
\end{remark}

\section*{Aknowledgements}
This work will appear in the PhD thesis of the first author, who was supported by the Woolf Fisher Trust and the Cambridge Trust. We are grateful to Robert Spencer for his help computing the inner products of polytabloids in Sage.

\thispagestyle{footer}

\bibliographystyle{abbrv}

\end{document}